\documentclass[reqno,10pt,centertags]{amsart}
\usepackage{amsmath,amsthm,amscd,amssymb,latexsym,esint,upref,stmaryrd,enumerate,verbatim,yfonts,mathrsfs}
\usepackage{color}
\usepackage{hyperref}

\newcommand*{\mailto}[1]{\href{mailto:#1}{\nolinkurl{#1}}}
\newcommand{\arxiv}[1]{\href{http://arxiv.org/abs/#1}{arXiv:#1}}

\newcommand{\bbN}{{\mathbb{N}}}

\newcommand{\bbR}{{\mathbb{R}}}

\newcommand{\cH}{{\mathcal H}}

\newcommand{\cN}{{\mathcal N}}

\newcommand{\cW}{{\mathcal W}}

\DeclareMathOperator{\dom}{dom}

\renewcommand{\Im}{\text{\rm Im}}

\newcommand{\lb}{\label}

\newcommand{\bi}{\bibitem}

\newcommand{\Om}{\Omega}

\let\geq\geqslant
\let\leq\leqslant

\makeatletter
\def\theequation{\@arabic\c@equation}

\allowdisplaybreaks 
\numberwithin{equation}{section}

\newtheorem{theorem}{Theorem}[section]

\newtheorem{lemma}[theorem]{Lemma}

\newtheorem{hypothesis}[theorem]{Hypothesis}

\theoremstyle{remark}

\begin{document}

\numberwithin{equation}{section}
\allowdisplaybreaks

\title[The Krein-von Neumann Realization of $-\Delta+V$ on Lipschitz Domains]{The Krein-von Neumann Realization of Perturbed Laplacians on Bounded Lipschitz Domains}

\author[J.\ Behrndt]{Jussi Behrndt}  
\address{Institut f\"ur Numerische Mathematik, Technische Universit\"at Graz, Steyrergasse 30, 8010 Graz, Austria}  
\email{behrndt@tugraz.at}  
\urladdr{http://www.math.tugraz.at/~behrndt/}
  
\author[F.\ Gesztesy]{Fritz Gesztesy}  
\address{Department of Mathematics,
University of Missouri, Columbia, MO 65211, USA}
\email{gesztesyf@missouri.edu}
\urladdr{http://www.math.missouri.edu/personnel/faculty/gesztesyf.html}

\author[T.\ Micheler]{Till Micheler}
\address{Department of Mathematics, Technische Universit\"{a}t Berlin, 10623 Berlin, Germany} 
\email{micheler@math.tu-berlin.de}

\author[M.\ Mitrea]{Marius Mitrea}
\address{Department of Mathematics,
University of Missouri, Columbia, MO 65211, USA}
\email{mitream@missouri.edu}
\urladdr{http://www.math.missouri.edu/personnel/faculty/mitream.html}

\thanks{J.\,B.\ gratefully acknowledges financial support by the Austrian
Science Fund (FWF), project P 25162-N26.} 
\thanks{Work of M.\,M.\ was partially supported by the Simons Foundation Grant $\#$\,281566
and by a University of Missouri Research Leave grant.} 

\date{\today}

\begin{abstract} 
In this paper we study the self-adjoint Krein--von Neumann realization $A_K$ of the perturbed Laplacian 
$-\Delta+V$ in a bounded Lipschitz domain $\Omega\subset\mathbb{R}^n$. We 
provide an explicit and self-contained description of the domain of $A_K$ in terms of Dirichlet and Neumann boundary traces,
and we establish a Weyl asymptotic formula for the eigenvalues of $A_K$.
\end{abstract}

\maketitle

\section{Introduction}  \lb{s1}

The main objective of this note is to investigate the self-adjoint Krein--von Neumann realization associated to the differential expression $-\Delta+V$
in $L^2(\Omega)$, where $\Omega\subset\mathbb{R}^n$, $n>1$, is assumed to be a bounded Lipschitz domain and $V$ is a nonnegative bounded
potential. In particular, we obtain an explicit description of the domain of $A_K$ in terms of Dirichlet and Neumann boundary traces,
and we prove the Weyl asymptotic formula
\begin{equation}\label{introweyl}
N(\lambda,A_K)\underset{\lambda\to\infty}{=}(2\pi)^{-n}v_n|\Omega|\,\lambda^{n/2} 
+O\big(\lambda^{(n-(1/2))/2}\big).
\end{equation}
Here $N(\lambda,A_K)$ denotes the number of nonzero eigenvalues of $A_K$ 
not exceeding $\lambda$, $v_n$ is the volume of the unit ball in $\mathbb{R}^n$, 
and $|\Omega|$ is the ($n$-dimensional) Lebesgue measure of $\Omega$. 

Let us first recall the definition and some properties of the Krein--von Neumann extension in the abstract setting.
Let $S$ be a closed, densely defined, symmetric operator in a Hilbert space $\mathcal H$  and assume that $S$ is strictly positive, that is, for some $c>0$, $(Sf,f)_\cH\geq c\Vert f\Vert^2_\cH$  
for all $f\in\dom (S)$. The Krein--von Neumann extension $S_K$ of $S$ is then given by
\begin{equation}\label{sk}
 S_K f=S^* f,\quad f\in\dom (S_K)=\dom (S)\,\dot+\,\ker (S^*),
\end{equation}
see the original papers Krein \cite{Kr47} and von Neumann \cite{N29}. 
It follows that $S_K$ is a nonnegative self-adjoint extension of $S$ and that for all other nonnegative self-adjoint extensions $S_\Theta$ of $S$ 
the operator inequality $S_K\leq S_\Theta$ holds in the sense of quadratic forms. As 
$\ker (S_K)=\ker (S^*)$, it is clear that $0$ is an eigenvalue of $S_K$ 
(except if $S$ is self-adjoint, in which case $S_K=S^*=S$).
Furthermore, if the self-adjoint Friedrichs extension $S_F$ of $S$ has purely discrete spectrum then the same is true for the spectrum of $S_K$ with the
possible exception of the eigenvalue $0$, which may have infinite multiplicity. For further developments, extensive references, and a 
more detailed discussion of the properties of the Krein--von Neumann extension of a symmetric operator we refer the reader to \cite[Sect.~109]{AG81a}, 
\cite{AS80}, \cite{AN70}--\cite{Ar00}, \cite[Chs.\ 9, 10]{ABT11}, \cite{AHSD01}--\cite{AT09}, 
\cite{AGMT10}, \cite{AGMST10}, \cite{AGMST13}, \cite{BGMM15}, \cite{Bi56}, \cite{DM91}, \cite{DM95}, \cite[Sect.~15]{Fa75}, \cite[Sect.\ 3.3]{FOT11}, \cite{GM11}, \cite{Gr83}, \cite[Sect.~13.2]{Gr09}, \cite{Gr12}, \cite{HMS04}, \cite{Ma92}, \cite{Ne83}, \cite{PS96}, 
\cite[Ch.~13]{Sc12}, \cite{SS03}, \cite{Sk79}, \cite{St96},  \cite{Ts80}, \cite{Ts81}, \cite{Ts92}, 
and the references cited therein. 

In the concrete case considered in this paper, the symmetric operator $S$ above is given by the minimal operator $A_{min}$ associated to the differential
expression $-\Delta+V$ in the Hilbert space $L^2(\Omega)$, that is,
\begin{equation}\label{aminintroqqq}
 A_{min}=-\Delta +V,\quad \dom (A_{min})=\mathring{H}^2(\Omega),
\end{equation}
where $\mathring{H}^2(\Omega)$ denotes the closure of $C_0^\infty(\Omega)$ in the Sobolev 
space $H^2(\Omega)$, and $0 \leq V \in L^{\infty}(\Omega)$. 
It can be shown that $A_{min}$ is the closure
of the symmetric operator $-\Delta +V$ defined on $C_0^\infty(\Omega)$. We point out that here $\Omega$ is a bounded Lipschitz domain and no further
regularity assumptions on $\partial\Omega$ are imposed, thus it is remarkable that the functions in $\dom (A_{min})$ possess $H^2$-regularity. The adjoint
$A_{min}^*$ of $A_{min}$ coincides with the maximal operator
\begin{equation}
 A_{max}=-\Delta +V,\quad \dom (A_{max})
 =\bigl\{f\in L^2(\Omega) \, \big| -\Delta f+Vf\in L^2(\Omega)\bigr\},
\end{equation}
where $\Delta f$ is understood in the sense of distributions. From 
\eqref{sk} and \eqref{aminintroqqq} it is clear that the Krein--von Neumann extension $A_K$ of $A_{min}$ is then given by
\begin{equation}\label{akintro}
 A_K =-\Delta +V,\quad \dom (A_K)=\mathring{H}^2(\Omega)\,\dot+\,\ker (A_{max}).
\end{equation}
In the present situation $A_{min}$ is a symmetric operator with infinite defect indices 
and therefore $\ker(A_{min}^*)=\ker(A_{max})$ 
is infinite-dimensional. In particular, $0$
is an eigenvalue of $A_K$ with infinite multiplicity, and hence belongs to the essential spectrum. It is also important to note that in general the functions 
in $\ker (A_K)$ do not possess any Sobolev regularity, that is, $\ker (A_K)\not\subset H^s(\Omega)$ for every $s>0$.
Moreover, since $\Omega$ is a bounded set,  
the Friedrichs extension of $A_{min}$ (which coincides with the self-adjoint Dirichlet operator associated to $-\Delta+V$) has compact resolvent and 
hence its spectrum is discrete. The abstract considerations above then yield that with the exception of the eigenvalue $0$
the spectrum of $A_K$ consists of a sequence of positive eigenvalues with finite multiplicity which tends to $+\infty$. 

The description of the domain of the Krein--von Neumann extension $A_K$ in \eqref{akintro} is not satisfactory for applications
involving boundary value problems. Instead, a more explicit description of $\dom(A_K)$
via boundary conditions seems to be natural and desirable. In the case of a bounded $C^\infty$-smooth domain $\Omega$, it is known that
\begin{equation}\label{akakintro}
 \dom (A_K)=\bigl\{f\in \dom(A_{max}) \, \big| \, \gamma_N f+M(0)\gamma_D f=0\bigr\}
\end{equation}
holds, where $\gamma_D$ and $\gamma_N$ denote the Dirichlet and Neumann trace operator, respectively, defined on the maximal domain $\dom(A_{max})$,
and $M(0)$ is the Dirichlet-to-Neumann map or Weyl--Titchmarsh operator for $-\Delta +V$. The description \eqref{akakintro} goes back to Grubb \cite{Gr68}, where certain classes of elliptic 
differential operators with smooth coefficients are discussed in great detail. Note that 
in contrast to the Dirichlet and Neumann boundary conditions the boundary condition
in \eqref{akakintro} is nonlocal, as it involves $M(0)$ which, when $\Om$ is smooth, is a boundary
pseudodifferential operator of order $1$. 
It is essential for the boundary condition \eqref{akakintro} that both trace operators 
$\gamma_D$ and $\gamma_N$ are defined on $\dom (A_{max})$. Even in the case of a smooth boundary $\partial\Omega$, the elements in $\dom(A_K)$,
in general, do not possess any $H^s$-regularity for $s>0$, and hence special attention has to 
be paid to the definition and the properties
of the trace operators. In the smooth setting the classical analysis due to Lions and Magenes \cite{LM} ensures that 
$\gamma_D:\dom (A_{max})\rightarrow H^{-1/2}(\partial\Omega)$ and 
$\gamma_N:\dom (A_{max})\rightarrow H^{-3/2}(\partial\Omega)$ are well-defined
continuous mappings when $\dom (A_{max})$ is equipped with the graph norm. 

Let us now turn again to the present situation, where $\Omega$ is assumed to be a bounded Lipschitz domain. Our first main objective is to 
extend the description of $\dom (A_K)$ in \eqref{akakintro} to the nonsmooth setting. The main difficulty here is to define appropriate trace operators
on the domain of the maximal operator. We briefly sketch the strategy from \cite{BM14}, which is mainly based and inspired by abstract extension theory
of symmetric operators. For this denote by $A_D$ and $A_N$ the self-adjoint realizations of $-\Delta+V$ corresponding to Dirichlet and Neumann boundary conditions, respectively. Recall that by \cite{JK95} and \cite{FMM98}
their domains $\dom(A_D)$ and $\dom (A_N)$ are both contained in $H^{3/2}(\Omega)$. Now
consider the boundary spaces 
\begin{equation}
 \mathscr G_D(\partial\Omega):=\bigl\{\gamma_D f \, \big| \, f\in\dom (A_N)\bigr\},\quad
 \mathscr G_N(\partial\Omega):=\bigl\{\gamma_N f \, \big| \, f\in\dom (A_D)\bigr\},
 \end{equation}
equipped with suitable inner products induced by the Neumann-to-Dirichlet map and Dirichlet-to-Neumann map for $-\Delta +V-i$, see Section~\ref{sec3}
for the details. It turns out that  $\mathscr G_D(\partial\Omega)$ and 
$\mathscr G_N(\partial\Omega)$ are both Hilbert spaces which are densely embedded in $L^2(\partial\Omega)$. It was shown in \cite{BM14} that 
the Dirichlet trace operator $\gamma_D$ and Neumann trace operator $\gamma_N$  can be extended
 by continuity to surjective mappings 
 \begin{equation}
  \widetilde\gamma_D:\dom (A_{max})\rightarrow \mathscr G_N(\partial\Omega)^*\quad\text{and}\quad
  \widetilde\gamma_N:\dom (A_{max})\rightarrow \mathscr G_D(\partial\Omega)^*,
 \end{equation}
 where $\mathscr G_D(\partial\Omega)^*$ and $\mathscr G_N(\partial\Omega)^*$ denote the adjoint (i.e., conjugate dual) spaces of $\mathscr G_D(\partial\Omega)$ and 
$\mathscr G_N(\partial\Omega)$, respectively. Within the same process also the Dirichlet-to-Neumann map $M(0)$ of $-\Delta+V$ (originally defined as a mapping from
$H^1(\partial\Omega)$ to $L^2(\partial\Omega)$) admits an extension to a mapping $\widetilde M(0)$ from  $\mathscr G_N(\partial\Omega)^*$ to
$\mathscr G_D(\partial\Omega)^*$. With the help of the trace maps $\widetilde\gamma_D$ and $\widetilde\gamma_N$, and the extended Dirichlet-to-Neumann 
operator $\widetilde M(0)$ we are then able to extend the description of the domain of the Krein--von Neumann extension for smooth domains in \eqref{akakintro} to the
case of Lipschitz domains. More precisely, we show in Theorem~\ref{akthm} that the Krein--von Neumann extension $A_K$ of $A_{min}$ is defined on 
 \begin{equation}\label{akakintro2}
 \dom (A_K)=\bigl\{f\in \dom(A_{max}) \, \big| \, \widetilde\gamma_N f+\widetilde M(0)\widetilde \gamma_D f=0\bigr\}.
\end{equation}
For an exhaustive treatment of boundary trace operators on bounded Lipschitz domains in 
$\bbR^n$ and applications to Schr\"odinger operators we refer to \cite{BGMM15}.

Our second main objective in this paper is to prove the Weyl asymptotic formula \eqref{introweyl} for the nonzero eigenvalues of $A_K$. We mention that the study of
the asymptotic behavior of the spectral distribution function of the Dirichlet Laplacian originates in the work by Weyl (cf.\ \cite{We12a}, \cite{We12}, and the 
references in \cite{We50}), and that generalizations of the classical Weyl asymptotic formula were obtained in numerous papers - we refer the reader to
\cite{BS70}, \cite{BS71}, \cite{BS72}, \cite{BS73}, \cite{BS79}, \cite{BS80}, \cite{Da97}, \cite{NS05}, \cite{SV97}, and the introduction in \cite{AGMST13}
for more details. There are relatively few papers available that treat the spectral asymptotics of the Krein Laplacian or
the perturbed Krein Laplacian $A_K$. Essentially these considerations are inspired by Alonso 
and Simon who, at the end of their paper \cite{AS80} posed the question if the asymptotics of the nonzero eigenvalues of the Krein Laplacian is given by Weyl's formula?
In the case where $\Omega$ is bounded and $C^{\infty}$-smooth, and 
$V\in C^{\infty}(\overline{\Omega})$, this has been shown to be the case three years 
later by Grubb \cite{Gr83}, see also the more recent contributions \cite{Mi94}, \cite{Mi06}, and \cite{Gr12}.
Following the ideas in \cite{Gr83} it was shown in \cite{AGMT10} that for so-called quasi-convex domains (a nonsmooth subclass of bounded Lipschitz domains 
with the key feature that $\dom (A_D)$ and $\dom (A_N)$ are both contained in $H^2(\Omega)$)
the Krein--von Neumann extension $A_K$ is spectrally equivalent to the buckling of a clamped plate problem, which in turn can be reformulated
with the help of the quadratic forms
\begin{equation}\label{aformintro}
 \mathfrak a[f,g]:=\bigl(A_{min} f,A_{min} g\bigr)_{L^2(\Omega)} \, \text{ and } \, 
 \mathfrak t[f,g]:=\bigl( f,A_{min} g\bigr)_{L^2(\Omega)},
 \end{equation}
defined on $\dom (A_{min})=\mathring{H}^2(\Omega)$.
In the Hilbert space $(\mathring{H}^2(\Omega),\mathfrak a[\cdot,\cdot])$ the form $\mathfrak t$ can be expressed with the help of a 
nonnegative compact operator $T$, and it follows that 
\begin{equation}\label{schoenesache}
 \lambda\in\sigma_p(A_{K})\backslash\{0\}\,\text{ if and only if }\, 
\lambda^{-1}\in\sigma_p(T), 
\end{equation}
counting multiplicities. These considerations can be extended from quasi-convex domains to the more general setting of Lipschitz domains, see, for instance, 
Section~\ref{sec4} and Lemma~\ref{mainlem}. Finally, the main ingredient in the proof of the Weyl asymptotic formula
\eqref{introweyl} for the Krein--von Neumann extension $A_K$ of $-\Delta+V$ on a bounded Lipschitz domain $\Omega$ is then 
a more general Weyl type asymptotic formula due to 
Kozlov \cite{Ko83} (see also \cite{Ko79}, \cite{Ko84}) which yields the asymptotics of the spectral distribution of function of the 
compact operator $T$, and hence via \eqref{schoenesache} the asymptotics of the spectral distribution function of $A_K$.
This reasoning in the proof of our second main result Theorem~\ref{mainthm} is along the lines  of 
\cite{AGMT10,AGMST10}, where the special case of quasi-convex domains was treated. 
For perturbed Krein Laplacians this result completes work that started with Grubb more than 30 years ago and demonstrates 
that the question posed by Alonso and Simon in \cite{AS80} regarding the validity of the 
Weyl asymptotic formula continues to have an affirmative answer for  bounded Lipschitz domains -- the natural end of the line in the 
development from smooth domains all the way to minimally smooth ones.

\section{Schr\"{o}dinger Operators on Bounded Lipschitz Domains}  
\lb{s3}

This section is devoted to a study of self-adjoint Schr\"{o}dinger operators on a nonempty,  
bounded Lipschitz domain in $\bbR^n$ (automatically assumed to be open). We shall make 
the following general assumption.

\begin{hypothesis}\lb{h2.1} 
Let $n\in\bbN\backslash\{1\}$, assume that $\Omega\subset\mathbb R^n$ is a bounded Lipschitz domain, 
and suppose that $0 \leq V\in L^\infty(\Omega)$. 
\end{hypothesis}

We consider operator realizations of the differential expression $-\Delta+V$ in 
the Hilbert space $L^2(\Omega)$. For this we define the \textit{preminimal} realization $A_p$ 
of $-\Delta+V$ by 
\begin{equation}\lb{We-Q.1}
A_{p}:=-\Delta+V,\quad\dom(A_{p}):= C^\infty_0(\Omega).
\end{equation} 
It is clear that $A_{p}$ is a densely defined, symmetric operator in $L^2(\Omega)$, and hence 
closable. The \textit{minimal} realization $A_{min}$ of $-\Delta+V$ is defined 
as the closure of $A_{p}$ in $L^2(\Omega)$, 
\begin{equation}\lb{eqn:Amin1}
A_{min}:=\overline{A_p}.
\end{equation} 
It follows that $A_{min}$ is a densely defined, closed, symmetric operator in $L^2(\Omega)$. 
The \textit{maximal} realization $A_{max}$ of $-\Delta+V$ is given by 
\begin{equation}\lb{We-Q.2}
A_{max}:=-\Delta+V,\quad
\dom(A_{max}):=\bigl\{f\in L^2(\Omega)\, \big| -\Delta f+Vf\in L^2(\Omega)\bigr\},    
\end{equation}
where the expression $\Delta f$, $f\in L^2(\Omega)$, is understood in the sense of distributions. 

In the next lemma we collect some properties of the operators $A_{p}$, $A_{min}$, 
and $A_{max}$.
The standard $L^2$-based Sobolev spaces of order $s\geq 0$ 
will be denoted by $H^s(\Omega)$; for the closure of $C_0^\infty(\Omega)$ in $H^s(\Omega)$ we write $\mathring{H}^s(\Omega)$.

\begin{lemma}\lb{lemmaa}
Assume Hypothesis~\ref{h2.1} and let $A_{p}$, $A_{min}$, and $A_{max}$ 
be as introduced above. Then the following assertions hold:
\begin{itemize}
 \item [$(i)$] $A_{min}$ and $A_{max}$ are adjoints of each other, that is, 
 \begin{equation}\lb{eqn:AminAmax}
A_{min}^*=A_{p}^*=A_{max}\,\text{ and }\, 
A_{min}=\overline{A_{p}}=A_{max}^*.
\end{equation}
\item [$(ii)$] $A_{min}$ is defined on $\mathring{H}^2(\Omega)$, that is,
\begin{equation}\label{amindom}
 \dom(A_{min})=\mathring{H}^2(\Omega),
\end{equation}
and the graph norm of $A_{min}$ and the $H^2$-norm are equivalent on $\dom (A_{min})$.
\item [$(iii)$] $A_{min}$ is strictly positive, that  is, for some $C>0$ we have
\begin{equation}\label{aminsemi}
 (A_{min} f,f)_{L^2(\Omega)} \geq C \Vert f\Vert^2_{L^2(\Omega)},\quad f\in\mathring{H}^2(\Omega).
\end{equation}
\item [$(iv)$] $A_{min}$ has infinite deficiency indices.
\end{itemize}
\end{lemma}

One recalls that the {\it Friedrichs extension} $A_F$ of $A_{min}$ is defined by
\begin{equation}\label{afaf}
A_F:=-\Delta+V,\quad \dom(A_F):=\bigl\{f\in\mathring H^1(\Omega) \, \big| \, 
\Delta f\in L^2(\Omega)\bigr\}.
\end{equation}
It is well-known that $A_F$ is a strictly positive self-adjoint operator in $L^2(\Omega)$ with 
compact resolvent (see, e.g. \cite[Sect.~VI.1]{EE89}). 

In this note we are particularly interested in the {\em Krein--von Neumann extension} 
$A_K$ of $A_{min}$. According to \eqref{sk}, $A_K$ is given by
\begin{equation}\label{akjussi}
A_K:=-\Delta+V,\quad \dom(A_K):=\dom(A_{min})\,
\dotplus\,\ker(A_{max}). 
\end{equation}

In the following theorem we briefly collect some well-known properties of the Krein--von Neumann extension $A_{K}$ in the present setting. For more details we refer the reader to the celebrated paper \cite{Kr47} by Krein
and to \cite{AS80}, \cite{AN70}, \cite{AT03}, \cite{AGMT10}, \cite{AGMST10}, \cite{AGMST13}, \cite{Gr12}, and \cite{HMS04} for further developments and references.

\begin{theorem}\lb{t6.3}
Assume Hypothesis~\ref{h2.1} and let $A_{K}$ be the Krein--von Neumann extension of $A_{min}$. Then the following assertions hold: 
\begin{itemize}
 \item [$(i)$] $A_K$ is a nonnegative self-adjoint operator in $L^2(\Omega)$ and $\sigma(A_K)$ 
consists of eigenvalues only. The eigenvalue $0$ has infinite multiplicity, 
$$\dim(\ker(A_K))=\infty,$$ and the restriction $A_{K}|_{(\ker(A_{K}))^\bot}$ 
is a strictly positive self-adjoint operator in the Hilbert space $(\ker(A_{K}))^\bot$ 
with compact resolvent. 
\item [$(ii)$] $\dom (A_K)\not\subset H^s(\Omega)$ for every $s>0$.
\item [$(iii)$] A nonnegative self-adjoint operator $B$ in $L^2(\Omega)$ is a 
self-adjoint extension of $A_{min}$ if and only if  for some $($and, hence for all\,$)$ $\mu<0$, 
\begin{equation}\label{afakjussi}
 (A_{F} - \mu)^{-1}\leq (B-\mu)^{-1} \leq (A_{K}-\mu)^{-1}. 
\end{equation}
\end{itemize}
\end{theorem}

We also mention that the Friedrich extension $A_F$ and the Krein--von Neumann 
extension $A_K$ are relatively prime (or disjoint), that is, 
\begin{equation}\label{disjoint}
\dom(A_F)\cap\dom(A_K)=\dom(A_{min})=\mathring{H}^2(\Omega).
\end{equation}

For later purposes we briefly recall some properties of the Dirichlet and Neumann trace operator and 
the corresponding self-adjoint Dirichlet and Neumann realizations of $-\Delta+V$
in $L^2(\Omega)$. We consider the space
\begin{equation}
H^{3/2}_\Delta(\Omega):=\bigl\{f\in H^{3/2}(\Omega) \, \big| \,\Delta f\in L^2(\Omega)\bigr\}, 
\end{equation}
equipped with the inner product 
\begin{equation}
 (f,g)_{H^{3/2}_\Delta(\Omega)}=(f,g)_{H^{3/2}(\Omega)}+(\Delta f,\Delta g)_{L^2(\Omega)},
 \quad f,g\in H^{3/2}_\Delta(\Omega).
\end{equation}
One recalls that the Dirichlet and Neumann trace operators $\gamma_D$ and $\gamma_N$ 
defined by
\begin{equation}
 \gamma_D f:=f\upharpoonright_{\partial\Omega}\, \text{ and } \, 
 \gamma_Nf:=\mathfrak n\cdot\nabla f\upharpoonright_{\partial\Omega}, 
 \quad f\in C^\infty(\overline\Omega), 
\end{equation}
admit continuous extensions to operators
\begin{equation}\label{extensions}
 \gamma_D:H^{3/2}_\Delta(\Omega)\rightarrow H^1(\partial\Omega)\, \text{ and } \, \gamma_N:H^{3/2}_\Delta(\Omega)\rightarrow L^2(\partial\Omega).
\end{equation}
Here $H^1(\partial\Omega)$ denotes the usual $L^2$-based Sobolev space of order $1$ on $\partial\Omega$; 
cf.\ \cite[Chapter 3]{Mc00} and \cite{MM13}.
It is important to note that the extensions in \eqref{extensions} are both surjective, 
see \cite[Lemma 3.1 and Lemma 3.2]{GM11}.

In the next theorem we collect some properties of the Dirichlet realization $A_D$ and Neumann realization $A_N$  of $-\Delta +V$ in $L^2(\Omega)$. We recall that
the operators $A_D$ and $A_N$ are defined as the unique self-adjoint operators corresponding to the closed nonnegative forms
\begin{equation}
\begin{split}
 \mathfrak a_D[f,g]&:=(\nabla f,\nabla g)_{(L^2(\Omega))^n}+(Vf,g)_{L^2(\Omega)},
 \quad \dom (\mathfrak a_D) := \mathring H^1(\Omega),\\
 \mathfrak a_N[f,g]&:=(\nabla f,\nabla g)_{(L^2(\Omega))^n}+(Vf,g)_{L^2(\Omega)}, 
 \quad \dom (\mathfrak a_N) := H^1(\Omega).
 \end{split}
 \end{equation}
 In particular, one has $A_F=A_D$ by \eqref{afaf}. In the next theorem we collect some well-known facts about the self-adjoint operators $A_D$ and $A_N$.
 The $H^{3/2}$-regularity of the functions in their domains is remarkable, and a consequence of $\Omega$ being a bounded Lipschitz domain. We refer the reader
 to \cite[Lemma 3.4 and Lemma 4.8]{GM08} for more details, see also \cite{JK81,JK95} and \cite{FMM98}.  

\begin{theorem}\lb{tdn}
Assume Hypothesis~\ref{h2.1} and let $A_D$ and $A_N$ be the self-adjoint Dirichlet and Neumann realization of $-\Delta +V$ in $L^2(\Omega)$, respectively.
Then the following assertions hold:
\begin{itemize}
 \item [$(i)$] The operator $A_D$ coincides with the Friedrichs extension $A_F$ and is given by
 \begin{equation}
   A_D=-\Delta+V,\quad \dom (A_D)=\bigl\{f\in H^{3/2}_\Delta(\Omega) \, \big| \, 
   \gamma_D f=0\bigr\}.
 \end{equation}
 The resolvent of $A_D$ is compact, and the spectrum of $A_D$ is purely discrete and 
 contained in $(0,\infty)$.
\item [$(ii)$] The operator $A_N$ is given by 
\begin{equation}
   A_N=-\Delta+V,\quad \dom (A_N)=\bigl\{f\in H^{3/2}_\Delta(\Omega) \, \big| \,\gamma_N f=0\bigr\}.
 \end{equation}
The resolvent of $A_N$ is compact, and the spectrum of $A_N$ is purely discrete and contained in $[0,\infty)$.
\end{itemize}
\end{theorem}

\section{Boundary conditions for the Krein--von Neumann realization}\label{sec3}

Our goal in this section is to obtain an 
explicit description of the domain of the Krein--von Neumann extension $A_{K}$ in terms 
of Dirichlet and Neumann boundary traces. For this we describe an extension procedure of the trace maps $\gamma_D$ and $\gamma_N$ in \eqref{extensions} onto $\dom(A_{max})$ from \cite{BM14}.
We recall that for $\varphi\in H^1(\partial\Omega)$ and $z \in\rho(A_D)$, the boundary value problem
\begin{equation}\label{bvp}
 -\Delta f+Vf = z f,\quad \gamma_D f=\varphi,
\end{equation}
admits a unique solution $f_z(\varphi)\in H^{3/2}_\Delta(\Omega)$. Making use of this fact and the trace operators \eqref{extensions} we define the Dirichlet-to-Neumann operator 
 $M(z)$, $z \in \rho(A_D)$, as follows:
\begin{equation}\label{dnmap}
 M(z):L^2(\partial\Omega)\supset H^1(\partial\Omega)\rightarrow L^2(\partial\Omega),\quad \varphi\mapsto - \gamma_N f_z(\varphi),
\end{equation}
where  $f_z(\varphi)\in H^{3/2}_\Delta(\Omega)$ is the unique solution of \eqref{bvp}. It can be shown that $M(z)$ is an unbounded operator in $L^2(\partial\Omega)$.
Moreover, if $z\in\rho(A_D)\cap\rho(A_N)$ then $M(z)$ is invertible and 
the inverse $M(z)^{-1}$ is a bounded operator defined on $L^2(\partial\Omega)$. Considering 
$z=i$, we set
\begin{equation}
 \Sigma:=\Im\,(-M(i)^{-1}).
\end{equation}
The imaginary part
$\Im\, M(i)$ of $M(i)$ is a densely defined bounded operator in $L^2(\partial\Omega)$ and hence it admits a bounded closure 
\begin{equation}
 \Lambda:=\overline{\Im (M(i))}
\end{equation}
in $L^2(\partial\Omega)$. Both operators $\Sigma$ and $\Lambda$ are self-adjoint and invertible with unbounded inverses. Introducing the boundary spaces 
\begin{equation}
 \mathscr G_D(\partial\Omega):=\bigl\{\gamma_D f \, \big| \, f\in\dom (A_N)\bigr\}
 \end{equation}
 and
 \begin{equation}
 \mathscr G_N(\partial\Omega):=\bigl\{\gamma_N f \, \big| \, f\in\dom (A_D)\bigr\},
 \end{equation}
we equip $\mathscr G_D(\partial\Omega)$ and $\mathscr G_N(\partial\Omega)$ with the inner products
\begin{equation}
 (\varphi,\psi)_{\mathscr G_D(\partial\Omega)}:=\bigl(\Sigma^{-1/2}\varphi,\Sigma^{-1/2}\psi\bigr)_{L^2(\partial\Omega)},\quad\varphi,\psi\in \mathscr G_D(\partial\Omega),
\end{equation}
and
\begin{equation}
 (\varphi,\psi)_{\mathscr G_N(\partial\Omega)}:=\bigl(\Lambda^{-1/2}\varphi,\Lambda^{-1/2}\psi\bigr)_{L^2(\partial\Omega)},\quad\varphi,\psi\in \mathscr G_N(\partial\Omega),
\end{equation}
respectively. Then $\mathscr G_D(\partial\Omega)$ and $\mathscr G_N(\partial\Omega)$ both become Hilbert spaces which are dense in $L^2(\partial\Omega)$. The
corresponding adjoint (i.e., conjugate dual) spaces will be denoted by $\mathscr G_D(\partial\Omega)^*$ and $\mathscr G_N(\partial\Omega)^*$, respectively. 
The following result can be found in \cite[Section 4.1]{BM14}.

\begin{theorem}\label{bm14}
 Assume Hypothesis~\ref{h2.1}. Then the Dirichlet trace operator $\gamma_D$ and Neumann trace operator $\gamma_N$ in \eqref{extensions} can be extended
 by continuity to surjective mappings 
 \begin{equation}
  \widetilde\gamma_D:\dom (A_{max})\rightarrow \mathscr G_N(\partial\Omega)^* 
  \, \text{ and } \, 
  \widetilde\gamma_N:\dom (A_{max})\rightarrow \mathscr G_D(\partial\Omega)^*
 \end{equation}
such that $\ker(\widetilde\gamma_D)=\ker(\gamma_D)=\dom (A_D)$ and $\ker(\widetilde\gamma_N)=\ker(\gamma_N)=\dom (A_N)$.
\end{theorem}
In a similar manner the boundary value problem \eqref{bvp} can be considered for all $\varphi\in\mathscr G_N(\partial\Omega)^*$ and the Dirichlet-to-Neumann operator 
$M(\cdot)$ in \eqref{dnmap} can be extended. More precisely, the following statement holds.

\begin{theorem}\label{bm142}
 Assume Hypothesis~\ref{h2.1} and let $\widetilde\gamma_D$ and $\widetilde\gamma_N$ be the extended Dirichlet and Neumann trace operator from Theorem~\ref{bm14}.
 Then the following are true:
 \begin{itemize}
  \item [$(i)$] For $\varphi\in \mathscr G_N(\partial\Omega)^*$ and $z \in\rho(A_D)$ the boundary value problem
\begin{equation}\label{bvp2}
 -\Delta f+Vf= z f,\quad \widetilde\gamma_D f=\varphi,
\end{equation}
admits a unique solution $f_z(\varphi)\in \dom(A_{max})$.
\item [$(ii)$] For $z \in\rho(A_D)$ the Dirichlet-to-Neumann operator $M(z)$ in 
\eqref{dnmap} admits a continuous extension
\begin{equation}
 \widetilde M(z):\mathscr G_N(\partial\Omega)^*\rightarrow 
 \mathscr G_D(\partial\Omega)^*,\quad \varphi\mapsto -\widetilde\gamma_N f_z(\varphi),
\end{equation}
where  $f_z(\varphi)\in \dom(A_{max})$ is the unique solution of \eqref{bvp2}.
 \end{itemize}
\end{theorem}

Now we are able to state our main result in this section: A description of the domain of the 
Krein--von Neumann extension $A_K$ 
in terms of Dirichlet and Neumann boundary traces. The extended Dirichlet-to-Neumann map 
at $z = 0$ will enter as a regularization parameter
in the boundary condition. For $C^\infty$-smooth domains this result goes back to Grubb \cite{Gr68}, where a 
certain class of elliptic differential operators with smooth coefficients is discussed 
systematically. For the special case of a so-called quasi-convex domains Theorem~\ref{akthm} 
reduces to \cite[Theorem 5.5]{AGMST10} and \cite[Theorem 13.1]{GM11}. In an abstract setting the Krein--von Neumann extension appears in a similar form in \cite[Example 3.9]{BM14}. 

\begin{theorem}\label{akthm}
Assume Hypothesis~\ref{h2.1} and let $\widetilde\gamma_D$, $\widetilde\gamma_N$ and $\widetilde M(0)$ be as in Theorem~\ref{bm14} and Theorem~\ref{bm142}. 
Then the  Krein--von Neumann extension $A_K$ of 
$A_{min}$ is given by 
\begin{equation} 
A_K = - \Delta+V,    \quad 
 \dom(A_K) = \bigl\{f\in\dom(A_{max}) \, \big| \, 
\widetilde{\gamma}_N f+\widetilde M(0)\widetilde{\gamma}_D f=0\bigr\}.
\end{equation}
\end{theorem}
\begin{proof}   
We recall that the Krein--von Neumann extension $A_K$ of $A_{min}$ is defined on
\begin{equation}\label{akakdom1}
 \dom (A_K)=\dom(A_{min})\,
\dotplus\,\ker(A_{max}).
\end{equation}
Thus, from Lemma~\ref{lemmaa}\,$(ii)$ one concludes 
\begin{equation}\label{akakdom2}
 \dom (A_K)={\mathring H}^2(\Omega)\,
\dotplus\,\ker(A_{max,\Omega}).
\end{equation}
Next, we show the inclusion
\begin{equation}\label{bitteschoen}
\dom(A_K)\subseteq  
\bigl\{f\in\dom(A_{max}) \, \big| \, \widetilde \gamma_N f
+\widetilde M(0)\widetilde \gamma_D f=0\bigr\}.
\end{equation}
Fix $f\in\dom (A_K)$ and decompose $f$ in the form $f=f_{min}+f_0$, where 
$f_{min}\in\mathring{H}^2(\Omega)$ and $f_0\in\ker(A_{max})$ (cf.\ \eqref{akakdom2}). Thus, 
$\gamma_D f_{min}=\widetilde{\gamma}_D f_{min}=0$ and 
$\gamma_N f_{min}=\widetilde{\gamma}_N f_{min}=0$, and hence it follows from 
Theorem~\ref{bm142}\,$(ii)$ that 
\begin{equation}
\widetilde M(0)\widetilde{\gamma}_D f
=\widetilde M (0)\widetilde\gamma_D(f_{min}+f_0)
=\widetilde M (0)\widetilde\gamma_D f_0=     -\widetilde\gamma_N f_0=     -\widetilde\gamma_N f.
\end{equation} 
Thus, $\widetilde \gamma_N f+\widetilde M(0)\widetilde \gamma_D f=0$ and the inclusion \eqref{bitteschoen} holds.

Next we verify the opposite inclusion 
\begin{equation}\label{bitteschoen2}
\dom(A_K)\supseteq 
\bigl\{f\in\dom(A_{max}) \, \big| \,\widetilde \gamma_N f
+\widetilde M(0)\widetilde \gamma_D f=0\bigr\}.
\end{equation}
We use the direct sum decomposition
\begin{equation}\label{decoaa}
 \dom (A_{max})=\dom (A_D)\,\dot+\,\ker (A_{max}),
\end{equation}
which is not difficult to check.
Assuming that $f\in\dom(A_{max,\Omega})$ satisfies the boundary condition
\begin{equation}\label{bcbc}
\widetilde M(0)\widetilde\gamma_D f+\widetilde\gamma_N f=0, 
\end{equation}
according to the decomposition \eqref{decoaa} we write $f$ in the form 
$f=f_D+f_0$, where $f_D\in\dom(A_D)$ and $f_0\in\ker(A_{max})$. Thus,  
$\gamma_D f_D=\widetilde{\gamma}_D f_D=0$ by Theorem~\ref{bm14} and with the 
help of Theorem~\ref{bm142}\,$(ii)$ one obtains 
\begin{equation}
\widetilde M(0)\widetilde \gamma_D f
=\widetilde M(0)\widetilde \gamma_D(f_D+f_0)
=\widetilde M(0)\widetilde\gamma_D f_0=-\widetilde\gamma_N f_0.
\end{equation} 
Taking into account the boundary condition \eqref{bcbc}  one concludes 
\begin{equation}
 -\widetilde\gamma_N f=\widetilde M(0)\widetilde \gamma_D f=-\widetilde\gamma_N f_0, 
\end{equation}
and hence
\begin{equation}
0=\widetilde\gamma_N(f-f_0)=\widetilde\gamma_N f_D.
\end{equation}
Together with Theorem~\ref{bm14} this implies 
$f_D\in\ker (\widetilde\gamma_N)=\ker(\gamma_N)=\dom(A_N)$. Thus, one arrives at 
\begin{equation}\label{o765f4}
f_D\in\dom(A_D)\cap\dom(A_N)=\dom(A_{min})
=\mathring{H}^2(\Omega), 
\end{equation}
where \eqref{disjoint} and $A_D=A_F$ was used (cf.\ Theorem~\ref{tdn}\,$(i)$). Summing up, 
one has
\begin{equation}
 f=f_D+f_0\in \mathring{H}^2(\Omega)\,\dotplus\,\ker(A_{max})=\dom (A_K),
\end{equation}
which shows \eqref{bitteschoen2} and completes the proof of Theorem~\ref{akthm}.
\end{proof}

\section{Spectral asymptotics of the Krein--von Neumann extension}\label{sec4}

As the main result in this section we derive the following Weyl-type spectral 
asymptotics for the Krein--von Neumann extension $A_{K}$ of $A_{min}$. 

\begin{theorem}\lb{mainthm} 
Assume Hypothesis~\ref{h2.1}. Let $\{\lambda_j\}_{j\in\bbN}\subset(0,\infty)$ be the strictly 
positive eigenvalues of the  Krein--von Neumann extension $A_K$ enumerated in nondecreasing 
order counting multiplicity, and let
\begin{equation}\label{4455}
N(\lambda,A_K):=\#\bigl\{j\in\bbN : 0<\lambda_j\leq\lambda\bigr\}, \quad   \lambda>0,
\end{equation}
be the eigenvalue distribution function for $A_K$. 
Then the following Weyl asymptotic formula holds, 
\begin{equation}
N(\lambda,A_K)\underset{\lambda\to\infty}{=}\frac{v_n\,|\Omega|}{(2\pi)^n}\,\lambda^{n/2}
+O\big(\lambda^{(n-(1/2))/2}\big),
\end{equation} 
where $v_n = \pi^{n/2}/ \Gamma((n/2)+1)$ denotes the $($Euclidean\,$)$ volume of the unit ball in 
$\bbR^n$ $($with $\Gamma(\cdot)$ the classical Gamma function \cite[Sect.\ 6.1]{AS72}$)$ 
and $|\Omega|$ represents the $($$n$-dimensional$)$ Lebesgue measure of $\Omega$.
\end{theorem}

The proof of Theorem~\ref{mainthm} follows along the lines of 
\cite{AGMT10,AGMST10}, where the case of quasi-convex domains was investigated. The main ingredients are a general Weyl type asymptotic formula due to 
Kozlov \cite{Ko83} (see also \cite{Ko79}, \cite{Ko84} for related results) and the connection between the eigenvalues of the so-called buckling operator and
the positive eigenvalues of the Krein--von Neumann extension $A_K$ (cf.\ \cite{AGMST10}, 
\cite{AGMST13}).
We first consider the quadratic forms $\mathfrak a$ and $\mathfrak t$ on $\dom (A_{min})=\mathring{H}^2(\Omega)$ defined by 
\begin{align} 
& \mathfrak a[f,g]:=\bigl(A_{min} f,A_{min} g\bigr)_{L^2(\Omega)},\quad f,g\in\dom(\mathfrak a):=\mathring{H}^2(\Omega),   \label{aform} \\
& \mathfrak t[f,g]:=\bigl( f,A_{min} g\bigr)_{L^2(\Omega)},\quad f,g\in\dom(\mathfrak t):=\mathring{H}^2(\Omega).   \label{tform}
\end{align}
Since the graph norm of $A_{min}$ and the $H^2$-norm are equivalent on $\dom A_{min}=\mathring{H}^2(\Omega)$ by Lemma~\ref{lemmaa}\,$(ii)$,  
it follows that $\mathcal W:=(\dom(\mathfrak a);  (\cdot,\cdot)_{\mathcal W})$, where the inner product is defined by
\begin{equation}\label{cite1}
 (f,g)_{\mathcal W}:=\mathfrak a[f,g]=\bigl(A_{min} f,A_{min} g\bigr)_{L^2(\Omega)}, 
 \quad f,g\in \dom(\mathfrak a),
\end{equation}
is a Hilbert space. One observes that the embedding $\iota:\mathcal W\rightarrow L^2(\Omega)$ is compact; this is a consequence of $\Omega$ being bounded.  
Next, we consider for fixed $g\in\mathcal W$ the functional
\begin{equation}
 \mathcal W\ni f\mapsto \mathfrak t[\iota f, \iota g],
\end{equation}
which is continuous on the Hilbert space $\mathcal W$ and hence can be represented with the help of a bounded operator $T$ in $\mathcal W$ in the form
\begin{equation}\label{tdef}
 (f, Tg)_{\mathcal W}=\mathfrak t[\iota f, \iota g],\quad f,g\in\mathcal W.
\end{equation}
The nonnegativity of the form $\mathfrak t$ implies that $T$ is a self-adjoint and nonnegative operator in $\mathcal W$. Furthermore, one obtains from \eqref{tform} that
\begin{equation}
 (f, Tg)_{\mathcal W}=\mathfrak t[\iota f, \iota g]=\bigl( \iota f,A_{min} \iota g\bigr)_{L^2(\Omega)}=\bigl(f,\iota^* A_{min} \iota g\bigr)_{\mathcal W},\quad f,g\in\mathcal W,
\end{equation}
and hence,
\begin{equation}
 T=\iota^* A_{min} \iota.
\end{equation}
In particular, since $A_{min} \iota:\mathcal W\rightarrow L^2(\Omega)$ is defined on the whole space $\mathcal W$ and is closed as an operator from $\cW$ 
to $L^2(\Omega)$, it follows that $A_{min} \iota$ is bounded and
hence the compactness of $\iota^*:L^2(\Omega)\rightarrow\mathcal W$ implies that $T=\iota^* A_{min} \iota$ is a compact operator in the Hilbert space $\mathcal W$.

The next useful lemma shows that the eigenvalues of $T$ are precisely the reciprocals of 
the nonzero eigenvalues of $A_K$. Lemma~\ref{mainlem} is inspired by the connection of 
the Krein--von Neumann extension to the buckling of a clamped plate problem 
(cf.\ \cite[Theorem 6.2]{AGMST10} and \cite{AGMT10,AGMST13,Gr83}).

\begin{lemma}\label{mainlem}
Assume Hypothesis~\ref{h2.1} and let $T$ be the nonnegative compact operator in $\mathcal W$ defined by \eqref{tdef}. Then 
\begin{equation}
 \lambda\in\sigma_p(A_{K})\backslash\{0\}\,\text{ if and only if }\, 
\lambda^{-1}\in\sigma_p(T), 
\end{equation}
counting multiplicities.
\end{lemma}
\begin{proof}
Assume first that $\lambda\not=0$ is an eigenvalue of $A_K$ and let $g$ be a corresponding eigenfunction. We decompose $g$ in the form
\begin{equation}
 g=g_{min}+g_0,\quad g_{min}\in\dom (A_{min}),\,\,\, g_0\in\ker (A_{max}) 
\end{equation}
(cf.\ \eqref{akjussi}), where $g_{min}\not= 0$ as $\lambda\not=0$. Then one concludes  
\begin{equation}
 A_{min} g_{min}=A_K (g_{min} + g_0) =A_K g, 
\end{equation}
and hence, 
\begin{equation}
A_{min} g_{min} -\lambda g_{min}= A_K g-\lambda g_{min}=\lambda g-\lambda g_{min}=\lambda g_0\in\ker (A_{max}),
\end{equation}
so that
\begin{equation}
 A_{max} A_{min} g_{min} = \lambda A_{max} g_{min}= \lambda A_{min} g_{min}.
\end{equation}
This yields 
\begin{equation}\label{easycomp}
 \begin{split}
  (f,\lambda^{-1}  g_{min})_\cW&=\mathfrak a [f,\lambda^{-1}g_{min}]\\
  &=\bigl(A_{min} f,\lambda^{-1} A_{min}g_{min}\bigr)_{L^2(\Omega)} \\
  &=\bigl( f,\lambda^{-1} A_{max} A_{min}g_{min}\bigr)_{L^2(\Omega)}\\
  &= \bigl( f, A_{min} g_{min}\bigr)_{L^2(\Omega)} \\
  &= \mathfrak t[f, g_{min}]\\
  &=(f,   T g_{min})_{\cW}, \quad f\in\mathcal W, 
 \end{split}
\end{equation}
where, for simplicity, we have identified elements in $\mathcal W$ with those in 
$\dom (\mathfrak a)$, and hence omitted the embedding map $\iota$. 
From \eqref{easycomp} we then conclude
\begin{equation}\label{tkg}
 T g_{min}=\frac{1}{\lambda} g_{min},
\end{equation}
which shows that $\lambda^{-1}\in\sigma_p(T)$. 

Conversely, assume that $h\in\cW \backslash \{0\}$ and $\lambda\not=0$ are such that
\begin{equation}\label{tkg1}
 T h=\frac{1}{\lambda} h
\end{equation}
holds. Then it follows from \eqref{cite1} and \eqref{tdef} that
\begin{equation}
 \mathfrak a [f,h]=\mathfrak a [f,\lambda T h]= (f,\lambda Th)_\cW=\mathfrak t [f,\lambda h] =\big(f,\lambda A_{min} h\big)_{L^2(\Omega)}, \quad f\in\dom (\mathfrak a). 
\end{equation}
As a consequence of the first representation theorem for quadratic forms 
\cite[Theorem~VI.2.1\,$(iii)$, Example~VI.2.13]{Ka80}, one concludes that $A_{max} A_{min}$ 
is the representing operator for $\mathfrak a$, and therefore, 
\begin{equation}
 h\in \dom(A_{max} A_{min}) \, \text{ and } \, A_{max} A_{min}h=\lambda A_{min}h.
\end{equation}
In particular,  $h\in \dom(A_{min})$ and
\begin{equation}\label{kernaa}
\begin{split}
 A_{max} (A_{min}-\lambda)h&= A_{max} A_{min}h- \lambda A_{max}h\\
 &= A_{max} A_{min}h- \lambda A_{min}h\\
 &=0.
\end{split}
 \end{equation}
Defining 
\begin{equation}
 g:=\frac{1}{\lambda} A_{min} h= h+ \frac{1}{\lambda}\bigl(A_{min}-\lambda\bigr)h, 
\end{equation}
$h\in \dom(A_{min})$ and $(A_{min}-\lambda)h\in \ker (A_{max})$ by \eqref{kernaa}, 
together with \eqref{akjussi} imply $g\in\dom A_K$. Moreover, 
$g\not=0$ since $A_{min}$ is positive. Furthermore, 
\begin{equation}
 A_K g=A_{max}g=\frac{1}{\lambda} A_{max}A_{min} h = A_{min}h=\lambda g, 
\end{equation}
shows that $\lambda\in\sigma_p(A_K)$.
\end{proof}

\begin{proof}[Proof of Theorem~\ref{mainthm}]
Let $T$ be the nonnegative compact operator in $\mathcal W$ defined by \eqref{tdef}.
We order the eigenvalues of $T$ in the form 
\begin{equation}\lb{kko-7}
0\leq\cdots\leq\mu_{j+1}(T)\leq\mu_j(T)\leq\cdots\leq\mu_1(T),
\end{equation} 
listed according to their multiplicity, and set 
\begin{equation}\lb{kko-8}
\cN(\lambda,T):=\#\,\bigl\{j\in\bbN : \mu_j(T)\geq\lambda^{-1}\bigr\},  
\quad   \lambda>0.
\end{equation} 
It follows from Lemma~\ref{mainlem} that
\begin{equation}
 \cN(\lambda,T)=N(\lambda,A_K),\quad \lambda>0,
\end{equation} 
and hence \cite{Ko83}
yields the asymptotic formula, 
\begin{equation}
N(\lambda,A_K)=\cN(\lambda,T)\underset{\lambda\to\infty}{=} 
\omega\,\lambda^{n/2}+O\big(\lambda^{(n-(1/2))/2}\big),
\end{equation}
with 
\begin{align}\label{trray45y}
\omega &:=\frac{1}{n(2\pi)^n}\int_\Omega\Bigg(\int_{S^{n-1}} 
\Bigg[\frac{\sum_{j=1}^n\xi_j^2}{\sum_{j,k=1}^n\xi^2_j\xi^2_k}\Bigg]^{\frac{n}{2}}
\,d\omega_{n-1}(\xi)\Bigg)d^n x
\nonumber\\[4pt]
& \,\,=(2\pi)^{-n}\,v_n\,|\Omega|.
\end{align}
\end{proof}

For bounds on $N(\, \cdot \,,A_K)$ in the case of $\Omega \subset \bbR^n$ open and 
of finite ($n$-dimensional) Lebesgue measure, we refer to \cite{GLMS14}.


 
\end{document}